\documentclass[a4paper]{amsart}

\usepackage{amsaddr}

\usepackage{svninfo}

\usepackage{amsmath, amsthm, amsfonts, amssymb, latexsym, mathrsfs,stmaryrd}
\usepackage{mathtools}
\usepackage{accents}
\usepackage{cancel} 
\usepackage{hyperref}
\usepackage[utf8]{inputenc}
\usepackage{booktabs,colortbl}

\usepackage{enumitem}
\usepackage{hyperref}
\usepackage{graphicx}
\usepackage{subcaption}

\renewcommand{\today}{
  \ifcase\month\or
  January\or February\or March\or April\or May\or June\or
  July\or August\or September\or October\or November\or December\fi
  \space \number\year}

\newcommand\nmodels{\mathbin{\cancel{\models}}}

\newenvironment{fent}[1]{\buildrel\textstyle ^{#1}\over{\hbox{\vrule height-10pt depth5pt width-3pt}{\smash\models}}}

\newenvironment{nfent}[1]{\buildrel\textstyle ^{#1}\over{\hbox{\vrule height-10pt depth5pt
width-3pt}{\smash\nmodels}}}

\newenvironment{dotto}{\buildrel\textstyle.\over{\hbox{\vrule height3pt depth0pt width0pt}{\smash\to}}}

\newenvironment{lo}{{\L}ukasiewicz~}

\newtheorem{theorem}{Theorem}[section]

\newtheorem{lemma}[theorem]{Lemma}

\theoremstyle{definition}
\newtheorem{definition}[theorem]{Definition}

\newtheorem{example}[theorem]{Example}

\theoremstyle{remark}
\newtheorem{remark}[theorem]{Remark}

\DeclareSymbolFont{AMSb}{U}{msb}{m}{n}
\DeclareMathSymbol{\N}{\mathbin}{AMSb}{"4E}
\DeclareMathSymbol{\Z}{\mathbin}{AMSb}{"5A}
\DeclareMathSymbol{\R}{\mathbin}{AMSb}{"52}
\DeclareMathSymbol{\Q}{\mathbin}{AMSb}{"51}
\DeclareMathSymbol{\I}{\mathbin}{AMSb}{"49}
\DeclareMathSymbol{\C}{\mathbin}{AMSb}{"43}

\makeatletter

\def\dotminussym#1#2{%
  \setbox0=\hbox{$\m@th#1-$}%
  \kern.5\wd0%
  \hbox to 0pt{\hss\hbox{$\m@th#1-$}\hss}%
  \raise.8\ht0\hbox to 0pt{\hss$\m@th#1.$\hss}%
  \kern.5\wd0}

\makeatother

\title{Yet Another Generalization of The Notion of a Metric Space}

\author{Seyed Mohammad Amin Khatami}
\address{Department of Computer Science, Birjand University of Technology, Birjand, Iran}
\email{khatami@birjandut.ac.ir, amin\_khatami@aut.ac.ir} \urladdr{}
\author{Madjid Mirzavaziri}
\address{Department of Pure Mathematics, Ferdowsi University of Mashhad, P.O. Box 1159, Mashhad 91775, Iran}
\email{mirzavaziri@gmail.com, mirzavaziri@um.ac.ir} \urladdr{}

\begin{document}
\maketitle
\begin{abstract}
A generalization of the triangle inequality is introduced by a
mapping similar to a t-conorm mapping. This generalization leads us to
a notion for which we use the $\star$-metric terminology. We are interested in the topological space induced by a $\star$-metric. Considering some examples of non-trivial $\star$-metrizable topological spaces, we also study the product topology for a finite family of $\star$-metrizable topological spaces.
\smallskip
\noindent\emph{Keywords:} Metric space, Generalization of metric space, Metric topology
\end{abstract}
\section{Introduction.}
The familiar notion of a metric which seems to be introduced firstly
by the French mathematician Fr\'echet \cite{frechet}, is a mathematics model for distance.
A metric, which is expected as a distance mapping on a nonempty set $M$, is defined by a mapping $d:M^2\to[0,\infty)$ satisfying the following axioms:
\begin{enumerate}[label=(M\arabic*),ref=(M\arabic*),leftmargin=1cm]
\item\label{a1} (identity of indiscernibles) $\forall x\forall y\left(d(x,y)=0\Leftrightarrow x=y\right)$,
\item\label{a2} (symmetry) $\forall x\forall y\left(d(x,y)=d(y,x)\right)$,
\item\label{a3} (triangle inequality) $\forall x\forall y\forall z\left(d(x,y)\le d(x,z)+d(z,y)\right)$.
\end{enumerate}
There are different generalizations of the notion of a metric.
Pseudometric is a generalization in which the distance between two distinct points can be zero \cite{kelley}.
Metametric
is a distance in which identical points do not necessarily have zero distance \cite{meta}.
Quasimetric is defined by omitting the symmetric property of the metric mapping \cite{quasimetric}.
Semimetric is defined by omitting the triangle inequality \cite{semimetric}.
Ultrametric is a metric with the strong triangle inequality
$\forall x\forall y\forall z\left(d(x,y)\le \max\{d(x,z),d(z,y)\}\right)$ \cite{ultrametric}.
Probabilistic metric is a fuzzy generalization of a metric where the distance instead of
non-negative real numbers is defined on distribution functions \cite{menger, sh}.
There are many other extensions of the concept of a metric which have appeared in
literatures (e.g. see \cite{ceder, fuzzymetric, mirza, newmetric1}).

This paper is about a generalization of the notion of a metric, which is called a $\star$-metric, by spreading out the
triangle inequality. We use a symmetric associative nondecreasing continuous function $\star:[0,\infty)^2\to[0,\infty)$ with the boundary condition $a\star 0=a$ called
t-definer to extend the triangle inequality. The function $\star$ is indeed an
extension of a well-known function, namely t-conorm, to the set of non-negative real numbers.
Continuity of $\star$ implies the existence of a dual operator for it, called residua,
which simplify the calculations of $\star$-metric functions such as metric functions.
\section{$\star$-metric.}
Recall that a t-conorm is a symmetric associative binary operator on the closed unit interval
which is nondecreasing on both arguments satisfying $S(x,0)=0$ for all $x\in[0,1]$.
The following definition is an extension of the concept of a t-conorm.
\begin{definition}
A triangular definer or a t-definer is a function $\star:[0,\infty)^2\to[0,\infty)$ satisfying the
following conditions:
\begin{enumerate}[label=(T\arabic*),ref=(T\arabic*),leftmargin=1cm]
  \item\label{t1} $a\star b=b\star a$,
  \item\label{t2} $a\star(b\star c)=(a\star b)\star c$,
  \item\label{t3} $a\le b$ implies $a\star c\le b\star c$ and $c\star a\le c\star b$,
  \item\label{t4} $a\star 0=a$,
  \item\label{t5} $\star$ is continuous in its first component with respect to the Euclidean topology.
\end{enumerate}
\end{definition}
Obviously, because of the commutativity of a t-definer \ref{t1}, its continuity in the first component implies its continuity in the second component. Furthermore, \cite[Proposition 1.19]{kelement} shows that a t-definer is a non-decreasing function \ref{t3}, so the continuity in its first component is equivalent to its continuity.
\begin{definition}
Let $\star$ be a t-definer and $M$ be a nonempty set. A $\star$-metric on $M$
is a function $d:M^2\to[0,\infty)$ satisfies the first two axioms of metric, \ref{a1} and \ref{a2}, together with the $\star$-triangle inequality as follows.
\begin{enumerate}[label=(M3$\star$),ref=(M3$\star$),leftmargin=1.2cm]
\item\label{a4} ($\star$-triangle inequality) $\forall x\forall y\forall z\left(d(x,y)\le d(x,z)\star d(z,y)\right)$.
\end{enumerate}
In this case $(M,d)$ is called a $\star$-metric space. Additionally, if \ref{a1} is
changed in to the weak form
\begin{enumerate}[label=(M1'),ref=(M1'),leftmargin=1.15cm]
\item\label{a5} $\forall x\forall y\left(x=y \rightarrow d(x,y)=0\right)$,
\end{enumerate}
then $\star$ is called a $\star$-pseudometric.
\end{definition}
\begin{example}\label{exa}
The most important continuous t-conorms are \lo, Maximum, and Product t-conorms which are described by
\begin{center}
$S_L(a,b)=\min\{a+b,1\}$, $S_m(a,b)=\max\{a,b\}$, and
$S_\pi(a,b)=a+b-a.b$.
\end{center}
But a t-conorm is defined on the closed unit interval while a t-definer is defined
on non-negative real numbers. The most important t-definers are:
\begin{itemize}[leftmargin=0.5cm]
  \item \lo t-definer: $a\star_L b=a+b$,
  \item Maximum t-definer : $a\star_m b=\max\{a,b\}$.
\end{itemize}
Obviously, an $\star_L$-metric is actually a metric and
an $\star_m$-metric is an ultrametric.
\end{example}
The following example shows that there are $\star$-metrics which are not metric.
\begin{example}\label{notmetr}
Clearly $a\star_{p} b=(\sqrt{a}+\sqrt{b})^2$ is a t-definer.
The function $d(a,b)=(\sqrt{a}-\sqrt{b})^2$ forms an $\star_p$-metric on $[0,\infty)$ which
is not a metric. Indeed,
\begin{eqnarray*}
d(a,b)&=&\left(\sqrt{a}-\sqrt{b}\right)^2\\
&=&\left(\sqrt{a}-\sqrt{c}+\sqrt{c}-\sqrt{b}\right)^2\\
&\le&\Big(\sqrt{(\sqrt{a}-\sqrt{c})^2}+\sqrt{(\sqrt{c}-\sqrt{b})^2}\Big)^2\\
&=&\Big(\sqrt{d(a,c)}+\sqrt{d(c,b)}\Big)^2\\
&=&d(a,c)\star_p d(c,b),
\end{eqnarray*}
while $d(1,25)=16\nleq d(1,16)+d(16,25)=9+1$.
\end{example}
Note that one of the reasons that in Example \ref{notmetr} the function $d$ does not form
a metric is that $a\star_p b\nleq a+b$. The following definition describes a partial order
between t-definers.
\begin{definition}
Assume that $\star_1$ and $\star_2$ are two t-definers.
If the inequality $a\star_1 b\le a\star_2 b$ holds for all $a,b\ge 0$,
then $\star_1$ is called weaker than $\star_2$ (or $\star_2$ is called stronger than $\star_1$) and denoted by $\star_1\le \star_2$.
\end{definition}
\begin{remark}
The Maximum t-definer $\star_m$ is the weakest t-definer. Let $\star$ be an arbitrary t-definer and
$a,b\ge 0$. Since $a\ge 0$ we have $a\star b\ge b\star 0=b$. Similarly $a\star b\ge a$. So,
$a\star b\ge \max\{a,b\}$. It seems that the strongest t-definer can not be specified.
\end{remark}
\begin{example}\label{tdefinerexa}
Consider t-definers in Example \ref{exa} and \ref{notmetr}. Furthermore let $\star_s$ be defined
by $a\star_{s} b=\sqrt{a^2+b^2}$. Then we have
\begin{center}
$\max\{a,b\} \le a\star_s b\le a+b \le a\star_p b$.
\end{center}
\end{example}
\begin{remark}
Clearly, for any two t-definers $\star_1$ and $\star_2$, if
$\star_1\le\star_2$ then any $\star_1$-metric space is a $\star_2$-metric space.
In particular, an ultrametric space is a $\star$-metric space for any t-definer $\star$.
\end{remark}
The definition of the residuum of a t-conorm is the key point that we use t-conorm for
introducing the concept of t-definer. The residuum of a t-definer plays a role
such as the role of minus operator for addition operator.
\begin{definition}
Let $\star$ be a t-definer. The residuum of $\star$ is defined by
\begin{center}
$a\dotto b=\inf\{c: c\star a\ge b\}$.
\end{center}
Note that $b\in\{c: c\star a\ge b\}$ and therefore $\dotto$ uniquely exists. Furthermore,
for any $a, b, c\in[0,\infty)$,
\begin{equation}\label{residum}
c\ge a\dotto b \text{ if and only if }c\star a\ge b,
\end{equation}
which is called the residuation property of $\star$ and $\dotto$.
\end{definition}
\begin{lemma}\label{property}
Let $\star$ be a t-definer and $\dotto$ be its residuum. Then
\begin{enumerate}[leftmargin=0.5cm]
\item $a\dotto b=\min\{c: c\star a\ge b\}$,
\item $0\dotto a=a$,
\item $a\dotto b=0$ if and only if $a\ge b$,
\item $a\star(a\dotto b)=\max\{a,b\}$,
\item $a\dotto b\ge (a\dotto c)\dotto (c\dotto b)$.
\item $a\dotto b\le (a\dotto c)\star (c\dotto b)$.
\end{enumerate}
\end{lemma}
\begin{proof}
{\em (1)} Let $A=\{c: c\star a\ge b\}$ and $\alpha=\inf A$. So there exists a non-increasing sequence
$\{c_n\}\subseteq A$ such that $\lim c_n=\alpha$. Now, continuity of $\star$ in its
first component implies that
\begin{center}
$\alpha\star a=(\lim c_n) \star a=lim(c_n\star a)\ge b$.
\end{center}
So $\alpha\in A$ that is $\alpha=\min A$.

{\em (2)} $0\dotto a=\min\{c: c\star 0\ge a\}=\min\{c: c\ge a\}=a$.

{\em (3)} If $a\ge b$ then $a\star 0=a\ge b$. So $0\in\{c: c\star a\ge b\}$. Therefore $a\dotto b=0$. Conversely if $a\dotto b=0$ then $\min\{c: c\star a\ge b\}=0$ that is $a=0\star a\ge b$.

{\em (4)} If $a\ge b$, then by {\em (3)} $a\dotto b=0$ and therefore
$a\star(a\dotto b)=a\star 0=a$. If $a\le b$, then since $\star$ is a continuous function,
\begin{center}
$a\star(a\dotto b)=a\star \min\{c: c\star a\ge b\}=\min\{a\star c: a\star c\ge b\}$.
\end{center}
Now, taking the continuous function $f(c)=a\star c$ we have
\begin{center}
$f(0)=a\le b=b\star 0\le b\star a=f(b)$,
\end{center}
therefore by the intermediate value theorem there exists some $c\in[0,b]$ for which
$f(c)=b$. So, $\min\{a\star c: b\le a\star c\}=b$ that is $a\star(a\dotto b)=b$.

{\em (5)} Since $\star$ is commutative and associative, by {\em (4)} we get
\begin{eqnarray*}
  (a\dotto b)\star (b\dotto c)\star a &=& a\star (a\dotto b)\star (b\dotto c)\\
  &\ge& b\star (b\dotto c)\\
  &\ge& c
\end{eqnarray*}
Now, using the residuation property \ref{residum} two times on $(a\dotto b)\star (b\dotto c)\star a \ge c$ we get {\em (5)}.

{\em (6)} By {\em (4)} $a\dotto c\ge (c\dotto b)\dotto (a\dotto b)$. So, residuation of $\star$ and $\dotto$ fulfills the proof.
\end{proof}
\begin{example}\label{exa1}
Let's consider t-definers in Example \ref{tdefinerexa}. If $a\ge b$ then $a\dotto b=0$ and if $a<b$ then
\begin{itemize}[leftmargin=0.5cm]
\item[] for \lo t-definer "$+$": $a\dotto b=b-a$,
\item[] for Maximum t-definer: $a\dotto b=b$,
\item[] for $\star_s$: $a\dotto b=\sqrt{b^2-a^2}$,
\item[] for $\star_p$: $a\dotto b=(\sqrt{b}-\sqrt{a})^2$.
\end{itemize}
\end{example}
\begin{remark}\label{not}
Let $\star$ be a t-definer and $\dotto$ be its residuum. Define $d:[0,\infty)^2\to [0,\infty)$ by
$d(a,b)=(a\dotto b)\star (b\dotto a)$. Obviously, $d$ satisfies \ref{a2}.
Furthermore, Lemma \ref{property} (\em {3}) implies that $d$ satisfies \ref{a1} and
Lemma \ref{property} (\em {6}) implies that $d$ satisfies \ref{a4},
\begin{eqnarray*}
  d(a,b) &=& (a\dotto b)\star (b\dotto a) \\
  &\le& (a\dotto c)\star (c\dotto b)\star (b\dotto c)\star (c\dotto a)\\
  &=& (a\dotto c)\star (c\dotto a)\star (b\dotto c)\star (c\dotto b)\\
  &=& d(a,c)\star d(c,b).
\end{eqnarray*}
So, $d$ forms a $\star$-metric
on $[0,\infty)$. The induced $\star$-metrics of t-definers in Example \ref{tdefinerexa} are as follows
\begin{itemize}[leftmargin=0.5cm]
\item[] $d_L(a,b)=|b-a|$ forms an $\star_L$-metric on $[0,\infty)$,
\item[] $d_{max}(a,b)=\left\{
\begin{array}{ll}
0&a=b\\
\max\{a,b\}&a\ne b
\end{array}
\right.$ forms an $\star_m$-metric on $[0,\infty)$,
\item[] $d_s(a,b)=\sqrt{|b^2-a^2|}$ forms an $\star_s$-metric on $[0,\infty)$,
\item[] $d_p(a,b)=|\sqrt{b}-\sqrt{a}|^2$ forms an $\star_p$-metric on $[0,\infty)$.
\end{itemize}
Note that $d_L$ also defines a $\star_L$-metric (or a metric) on $\mathbb{R}$. Similarly,
$d_s$ also defines a $\star_s$-metric on $\mathbb{R}$.
\end{remark}

\section{Topology of $\star$-metric.}
In this section, we extend some topological concepts of metric spaces to $\star$-metric spaces.
\begin{definition}\label{openball}
Assume that $(M,d)$ is a $\star$-metric space. For any $a\in M$ and $r>0$,
the ``open ball around $a$ of radius $r$" is the set
\begin{center}
$N_r(a)=\{b: d(a,b)<r\}$.
\end{center}
For a subset $A$ of $M$, a point $x\in A$ is called an ``interior point" of $A$ if there
exists $\epsilon >0$ such that $N_\epsilon(x)\subset A$. $A$ is said to be
an ``open" subset of $M$ whenever any point of $A$ is an interior point.
\end{definition}
The following theorem shows that the set of all open subsets of a $\star$-metric space
$(M,d)$ forms a topology on $M$ called the $\star$-metric topology.
\begin{theorem}
For every $\star$-metric space $(M,d)$, the set of all open subsets of $M$ forms a Hausdorff topology on $M$, denoted by $\tau_d$.
\end{theorem}
\begin{proof}
Let $\tau_d=\{A\subseteq M: A~\text{is an open subset of}~M\}$. Obviously $\emptyset, M\in \tau_d$. Assume that
$A,B \in \tau_d$. Since $\emptyset\in \tau_d$, if $A\cap B=\emptyset$ there is nothing to prove. So, assume that
$A\cap B\ne\emptyset$. We indicate that any point $a\in A\cap B$ is an interior point of $A\cap B$. Since $A$ and $B$
are open sets, $a$ is an interior point of $A$ and $B$. So there exists $r>0$ and $s>0$ such that
$N_r(a)\subseteq A$ and $N_s(B)\subseteq B$. If we set $t=\min\{r,s\}$, then
$N_t(a)\subseteq N_r(a)\cap N_s(a)\subseteq A\cap B$ and so $a$ is an interior point of $A\cap B$.
On the other hand, an easy argument shows that the union of arbitrary family of open sets is open.

Now consider two distinct points $a, b\in M$. There
exists $s>$ such that $s\star s<d(a,b)$. Indeed, otherwise $d(a,b)<s\star s$ for any $s>0$ and therefore
continuity of $\star$ implies that $d(a,b)=0$ which is a contradiction. Now, we show
$N_s(a)\cap N_s(b)=\emptyset$ which completes the proof. To this end, if
there exists some $c\in N_s(a)\cap N_s(b)$ then we get the following contradiction:
\begin{center}
$d(a,b)\le d(a,c)\star d(c,b)< s\star s<d(a,b)$.
\end{center}
\end{proof}
The notions and concepts of topological spaces such as ``closed set", ``interior and closure of a set",
``limit point and the set of limit points of a set", ``continuous function", and so forth are defined as usual (e.g. see \cite{mun} or \cite{kelley}).

The following theorem shows that in $\star$-metric spaces, open balls are open sets .
\begin{lemma}\label{open}
In every $\star$-metric space $(M,d)$, open balls are open sets.
\end{lemma}
\begin{proof}
Assume that $\star$ is a t-definer, $\dotto$ is the residuum of $\star$,
$(X,d)$ is a $\star$-metric space, $x\in M$, and $r>0$. We show that every $y\in N_r(x)$ is
an interior point of $N_r(x)$. To this end for $\epsilon=d(x,y)\dotto r$, we show that
$N_\epsilon(y)\subseteq N_r(x)$. For this consider $z\in N_\epsilon(y)$. So ,
$d(z,y)<\epsilon$ that is $d(z,y) <d(x,y)\dotto r$.
Now the residuation of $\dotto$ and $\star$ implies that
$d(z,y)\star d(x,y)< r$ and so by the $\star$-triangle inequality and symmetric
property of $\star$ and $d$ we have
$d(x,z)\le d(x,y) \star d(y,z)<r$ which show that $z\in N_r(x)$.
\end{proof}
Now, by Definition \ref{openball} and Lemma \ref{open}, for a
$\star$-metric space $(M,d)$ the set, $\mathfrak{B}_d=\{N_r(a): a\in M \text{ and } r>0\}$ is a base for the induced topology of $d$ on $M$ which is called the open ball base of $\tau_d$.
\begin{theorem}
Every $\star$-metric space $(M,d)$ is first countable.
\end{theorem}
\begin{proof}
Let $a$ be an arbitrary point of $M$. We must show that there exists a countable family
$\{U_n\}_{n\in\mathbb{N}}$ of neighbourhoods of $a$ such that every neighbourhood of $a$
contain at least one of $U_n$s. To this end for any $n\in\mathbb{N}$
set $U_n=N_{1/n}(a)$. By Lemma \ref{open} any $U_n$ is a neighbourhood of $a$ and the proof is complete.
\end{proof}
\begin{theorem}
Every $\star$-metric space $(M,d)$ is Normal.
\end{theorem}
\begin{proof}
The proof is similar to the one for metric spaces (e.g. see \cite[Theorem 32.2]{mun}).
Let $A$ and $B$ be two closed subset of $M$. Since $B$ is a closed subset of $M$,
for any $a\in A$ let $N_{r_a}(a)$ be an open ball such that $N_{r_a}(a)\cap B=\emptyset$. Similarly, for any $b\in B$ the closeness of $A$ implies that one
could find $N_{r_b}(b)$ such that $N_{r_b}(b)\cap A=\emptyset$. Now, for
any $a\in A$ and $b\in B$ assume that $s_a$ and $s_b$ are such that
$s_a\star s_a<r_a$ and $s_b\star s_b<r_b$, respectively. Set,
\begin{center}
$U=\displaystyle\bigcup_{a\in A}N_{s_a}(a)$ and $V=\displaystyle\bigcup_{b\in B}N_{s_b}(b)$.
\end{center}
$U$ and $V$ are open sets containing $A$ and $B$ respectively.
Furthermore, we claim that $U\cap V=\emptyset$.
Indeed, if $c\in U\cap V$ then there exists $a\in A$ and $b\in B$ such that
$c\in N_{s_a}(a)\cap N_{s_b}(b)$ and therefore
$d(a,b)\le d(a,c)\star d(c,b)<s_a\star s_b$. Now, without loss of generality
we could assume that $s_b\le s_a$. So, $d(a,b)< s_a\star s_a<r_a$ which means
that $b\in N_{r_a}(a)$, a contradiction.
\end{proof}
\section{Product topology for $\star$-metric.}
Recall that for a family $\{(X_i,\tau_i)\}_{i\in I}$ of topological spaces, the product topology is
the weakest topology on $X=\prod_{i\in I}X_i$ which makes all of the projection
maps $\{\pi_i:X \to X_i\}_{\i\in I}$ continuous. Keep in mind that
$\{\prod_{i\in I}U_i: U_i$ is open in $X_i$ and $U_i\ne X_i$ for only finitely many $i\}$
is a base for the product topology on $X$. Furthermore, if for each $i\in I$ the topology on $X_i$ is given by a basis $\mathfrak{B}_i$ then
$\{\prod_{i\in I}B_i: B_i\in\mathfrak{B}_i$ and $B_i\ne X_i$ for only finitely many $i\}$
form a basis for product topology on $X$.
\begin{remark}\label{productmetric}
If $\{(M_i,d_i)\}_{i=1}^n$ be a finite family of metric spaces, then the
product topology on $M=\prod_{i=1}^n M_i$ is the same as the induced topology of
the following three significant metrics on $M=\prod_{i=1}^n M_i$
(e.g. see \cite[Theorem 4.5.1]{metricspace}).
\begin{itemize}[leftmargin=0.5cm]
  \item (Maximum metric) $d_{max}(\bar{x},\bar{y})=\max_{1\le i\le n}d_i(x_i,y_i)$,
  \item (Euclidean product metric) $d_E(\bar{x},\bar{y})=\sqrt{\sum_{i=1}^n d_i(x_i,y_i)^2}$,
  \item (Taxicab metric) $d_T(\bar{x},\bar{y})=\sum_{i=1}^n d_i(x_i,y_i)$,
\end{itemize}
The coming figure describes these three metrics and their corresponding open balls on $\mathbb{R}^2$ more precisely.
\begin{figure}[h!]
  \centering
  \begin{subfigure}[b]{0.59\linewidth}
    \includegraphics[width=\linewidth]{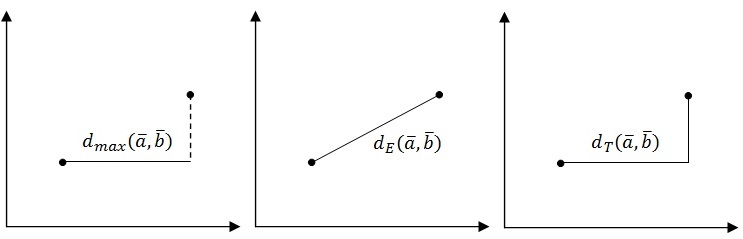}
  \end{subfigure}
  \begin{subfigure}[b]{0.4\linewidth}
    \includegraphics[width=\linewidth]{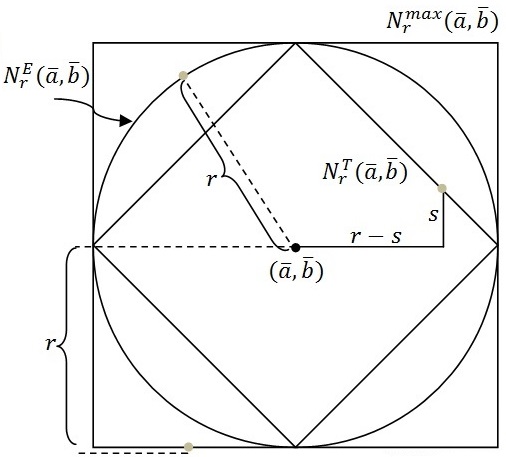}
  \end{subfigure}
  \caption{$d_{max}$, $d_E$, and $d_T$ on $\mathbb{R}^2$}\label{prod1}
\end{figure}
\end{remark}
\begin{definition}
For a $\star$-metric $d$ on $M$, the $\star$-product topology on $M^n$ is the product topology induced by the $\star$-metric topology of $M$.
\end{definition}
The following theorem demonstrates a situation similar to that of Remark \ref{productmetric} for $\star$-metric spaces.
\begin{theorem}
Let $\{(M_i,d_i)\}_{i=1}^n$ be a family of $\star$-metric spaces. Assume that $M=\prod_{1\le i\le n}M_i$ and define $d_{max}$ and $d_T$ by
\begin{itemize}
\item[] $d_{max}(\bar{x},\bar{y})=\max_{1\le i\le n}d_i(x_i,y_i)$,
\item[] $d_T(\bar{x},\bar{y})=d_1(x_1,y_1)\star d_2(x_2,y_2)\star...\star d_n(x_n,y_n)$.
\end{itemize}
Then $d_{max}$ and $d_T$ define $\star$-metrics on $M$. Furthermore the induced topology
of these two metrics on $M$ is the same as the product topology on $M$.
\end{theorem}
\begin{proof}
Obviously $d_{max}$ and $d_T$ satisfies the first two properties of $\star$-metric, namely
"identity of indiscernibles" and "symmetry".

For the $\star$-triangle inequality let
$\bar{x}, \bar{y}, \bar{z} \in M$. If
$\max_{1\le i\le n}d_i(x_i,y_i)=d_k(x_k, y_k)$ for some $1\le k\le n$,
then we have
\begin{eqnarray*}
d_{max}(\bar{x},\bar{y})&=&\max_{1\le i\le n}d_i(x_i,y_i)\\
&=&d_k(x_k, y_k)\\
&\le&d_k(x_k, z_k)\star d_k(z_k,y_k)\\
&\le&d_{max}(\bar{x},\bar{z})\star d_{max}(\bar{z},\bar{y}).
\end{eqnarray*}
The following argument shows that $d_T$ also admits the $\star$-triangle inequality.
\begin{equation*}
\begin{split}
d_T(\bar{x},\bar{y})&=d_1(x_1,y_1)\star d_2(x_2,y_2)\star...\star d_n(x_n,y_n)\\
&\le\big(d_1(x_1,z_1)\star d_1(z_1,y_1)\big)\star\big(d_2(x_2,z_2)\star d_2(z_2,y_2)\big)\star
...\star\big(d_n(x_n,z_n)\star d_n(z_n,y_n)\big)\\
&=\big(d_1(x_1,z_1)\star d_2(x_2,z_2)\star...\star d_n(x_n,z_n)\big)\star
\big(d_1(z_1,y_1)\star d_2(z_2,y_2)\star...\star d_n(z_n,y_n)\big)\\
&=d_T(\bar{x},\bar{z})\star d_T(\bar{z},\bar{y}).
\end{split}
\end{equation*}

For the latter argument, firstly note that the induced topology of $d_{max}$ on $M$ is as the same
as the induced topology of $d_T$ on $M$. Indeed, if we denote the elements of the
open ball base of induced topologies of $d_T$ and $d_{max}$ by
$N_r^T(\bar{a})$ and $N_r^{max}(\bar{a})$ respectively, then
\begin{center}
$N_r^T(\bar{a})\subseteq N_r^{max}(\bar{a})\subseteq N_{\underbrace{r\star r\star ...\star r}_{\text{n-times}}}^T(\bar{a})$.
\end{center}
Now, let $\mathfrak{B}$ be a basis for the product topology and
$B=\prod_{i=1}^n N_{r_i}(a_i)$ be an element of $\mathfrak{B}$ and
$\bar{x}\in B$. Since for each $1\le i\le n$, $x\in N_{r_i}(a_i)$ and
$N_{r_i}(a_i)$ is a $\star$-open set, there exists $\epsilon_i>0$ such that
$N_{\epsilon_i}(x_i)\subseteq N_{r_i}(a_i)$. Let $\epsilon=\min_{1\le i\le n}\{\epsilon_i\}$.
Obviously, $N_\epsilon^{max}(\bar{x})\subseteq B$.

On the other hand, if $\bar{x}\in N_r^{max}(\bar{x})$ then for each
$1\le i\le n$, $x_i\in N_r(a_i)$ and so there exists $\epsilon_i>0$ such that
$N_{\epsilon_i}(x_i)\subseteq N_r(a_i)$. Assuming $\epsilon=\min_{1\le i\le n}\{\epsilon_i\}$
we get $N_\epsilon(\bar{x})\subseteq N_r^{max}(\bar{a})$. So the induced topology of $d_{max}$
is as the same as the product topology on $M$.
\end{proof}


\end{document}